\documentclass[[sn-mathphys,Numbered]{svjour3}		

\usepackage{epsfig}

\graphicspath{{fig/}}

\usepackage{tikz}
\usepackage{geometry}
\usetikzlibrary{matrix,positioning,decorations.pathreplacing}

\makeatletter
\let\old@ssect\@ssect 
\makeatother

\usepackage{hyperref}

\makeatletter
\def\@ssect#1#2#3#4#5#6{%
  \NR@gettitle{#6}
  \old@ssect{#1}{#2}{#3}{#4}{#5}{#6}
}
\makeatother

\usepackage{amssymb,amsfonts,amsmath}%

\usepackage{enumitem}

\usepackage{color}
\usepackage{url}

\renewenvironment{proof}{{\bf Proof.}}{\hfill \hspace*{1pt}\hfill $\Box$}

\newcommand{\calK}{\mathcal{K}}

\newcommand{\Rpn}{\R_{0}^{+}}
\newcommand{\R}{\mathbb{R}}
\newcommand{\C}{\mathbb{C}}
\newcommand{\Z}{\mathbb{Z}}

\newcommand{\lel}{\left\langle}
\newcommand{\rir}{\right\rangle}
\newcommand{\scalp}[2]{ \lel #1, #2 \rir }

\newcommand{\esssup}{\operatorname*{ess\;sup}}%
\newcommand{\tm}{\times}%
\newcommand{\N}{\mathbb{N}}%

\newcommand{\amc}[1]{{\color{black} #1}}

\newcommand{\fsc}[1]{{\color{black} #1}}

\newcommand \Iff   {\Leftrightarrow}

\begin{document}

\title{Quadratic ISS Lyapunov functions for linear analytic systems}

\titlerunning{ Quadratic converse ISS Lyapunov theorems}


%
 %
%

\author{Andrii~Mironchenko* \and Felix L.~Schwenninger} 

\institute{A. Mironchenko is with Department of Mathematics, University of Bayreuth, Germany. Corresponding author. \email andrii.mironchenko@uni-bayreuth.de
\\
{F. Schwenninger is with Department of Applied Mathematics, University of Twente, P.O.~Box 217, 7500AE Enschede, The Netherlands.}
\email{f.l.schwenninger@utwente.nl}}

\maketitle

\begin{abstract}
\amc{
We derive converse Lyapunov theorems for input-to-state stability (ISS) of linear 
infinite-dimensional analytic systems. While we show that ISS in general does not imply the
existence of a coercive quadratic ISS Lyapunov function, even if the input operator is bounded, we prove that indeed  quadratic ISS Lyapunov functions always exist for $p$-admissible input operators with $p<2$, provided the semigroup is similar to a contraction on a Hilbert space.
The constructions are semi-explicit and rely on classical results on analytic semigroups and similarity to contractive ones. In the case of self-adjoint generators, they coincide with the canonical Lyapunov function being the norm squared. 
}
\end{abstract}

\keywords{
infinite-dimensional systems, linear systems, nonlinear systems,
input-to-state stability, Lyapunov methods
}
\subclass{37B25, 37C75, 93C25, 93D09}
\maketitle
%
%
%
%

\section{Introduction}

Input-to-state stability (ISS) was introduced by Sontag in his celebrated paper \cite{Son89} and has rapidly become a backbone of robust nonlinear control theory with diverse applications to robust stabilization \cite{FrK08}, nonlinear observer design \cite{AnP09}, analysis of large-scale networks \cite{JTP94,DRW10}, event-based control \cite{Tab07}, networked control systems \cite{NeT04}, ISS feedback redesign \cite{Son89}, quantized control \cite{LiH05}, nonlinear detectability \cite{KSW01}, etc. 
We refer to a recent monograph \cite{Mir23} for a detailed treatment of the classical ISS theory and control applications.
In the past decade, the ISS concept has been extended to broad classes of infinite-dimensional systems, including partial differential equations (PDEs) with distributed and boundary controls, semilinear evolution equations in abstract spaces, time-delay systems, etc. \cite{MiW18b,TPT18,JLR08,KaK16b,JNP18,JSZ17,ZhZ18,ZhZ20,KaK19}.
We refer to \cite{MiP20} for a survey of the state of the art of infinite-dimensional ISS theory and its applications to robust control and observation of distributed parameter systems, as well as to \cite{Sch20} for an overview of available results on ISS of linear boundary control systems and some semilinear extensions.

One of the most fundamental concepts of the ISS theory is the notion of an \emph{ISS Lyapunov function}. 
For ODE systems with Lipschitz continuous right-hand side, it is known that the existence of a coercive ISS Lyapunov function is equivalent to ISS \cite{SoW95}, which was extended to classes of semilinear evolution equations with Lipschitz nonlinearities and \emph{distributed} inputs in \cite{MiW17c}.
However, the application of Lyapunov methods for ISS analysis of boundary control systems leads to challenging theoretical problems.

For example, it is well-known that the classic heat equation with Dirichlet boundary inputs is ISS. 
However, no coercive ISS Lyapunov function is known for this system, and the existence of such a function is neither proved nor disproved. This challenge led to the thriving of non-Lyapunov methods for ISS analysis: admissibility theory \cite{JNP18} and spectral analysis \cite{KaK16b} have been successfully used for linear systems; a monotonicity approach \cite{MKK19} can be applied to monotone control systems such as nonlinear parabolic systems with distributed and boundary inputs (via maximum principle); De-Giorgi iteration \cite{ZhZ19} has been applied for (local) ISS analysis of Burgers' equation.
 Despite the efficiency of these methods within particular system classes, we still lack the generality of the Lyapunov method.

To address the problem with the applicability of Lyapunov tools, in \cite{MiW18b,MiW19a} the concept of a non-coercive Lyapunov function --- i.e.\ a Lyapunov function that is not necessarily coercive --- has been proposed, and in \cite{JMP20} it was shown that the existence of a non-coercive ISS Lyapunov function implies ISS, provided some additional properties of the system are satisfied. In particular, in \cite{JMP20}, a quadratic, non-coercive ISS Lyapunov function was constructed for the 1-D heat equation with a Dirichlet boundary input. 
Yet, it remains an open problem whether a coercive ISS Lyapunov function for such a simple boundary control system exists, which puts into question the applicability of Lyapunov methods for ISS analysis of boundary control systems. 
The problem becomes even more intriguing, as for linear parabolic systems with Neumann and Robin boundary inputs and for linear first-order hyperbolic systems (systems of conservation laws), rather simple coercive quadratic ISS Lyapunov functions exist \cite{TPT18,ZhZ18}. 

In \cite{Mar26}, a construction of coercive Lyapunov functions was proposed for systems that are regular, possess a non-coercive Lyapunov function, and for which there exists an exactly observable output. However, parabolic systems with Dirichlet inputs do not belong to this class. 
Another relaxation of the Lyapunov function concept, called generalized ISS Lyapunov function, has been proposed in \cite{ZhZ25}.

All this makes developing systematic Lyapunov methods for analysis of linear and nonlinear boundary control systems a central problem in the infinite-dimensional ISS theory.

\textbf{Contribution.}
\emph{In this paper, we give necessary and sufficient conditions for the existence of quadratic ISS Lyapunov functions for linear systems with bounded input operators, and prove the first converse ISS Lyapunov theorems for linear analytic systems with unbounded input operators. \fsc{The candidates for Lyapunov functions arise from rephrasing classical results on analytic semigroups, and when they are similar to contractive ones. While the related techniques are well-known and already present in infinite-dimensional systems theory, our contributions lie in the theory of ISS Lyapunov functions.   
For instance, while the coercivity of our Lyapunov functions is well-known from relations to the holomorphic functional calculus, it was not previously known for which control operators these  become ISS Lyapunov functions.}}
In \cite[Theorem 8]{MiW17c}, it was shown that a linear system on a general Banach space with distributed controls is ISS if and only if there exists a coercive ISS Lyapunov function, which is an equivalent norm on $X$. By taking squares, we can always find a 2-homogeneous Lyapunov function, which is, however, not quadratic in general, even if $X$ is a Hilbert space. Moreover, \cite[Theorem 8]{MiW17c} states that on Hilbert spaces, such a system is ISS if and only if there exists a non-coercive, quadratic ISS Lyapunov function. \emph{In Section~\ref{sec:Coercive quadratic Lyapunov functions for linear systems with bounded input operators}, we show that one cannot expect to have a coercive quadratic Lyapunov function even if additionally the system is supposed to be analytic. However, we also show that a coercive, quadratic Lyapunov function exists if the underlying semigroup is similar to a contraction semigroup on Hilbert space.}
\fsc{The latter property is well-studied in abstract Hilbert space semigroup theory and can, e.g.\ characterized by Callier--Grabowski--Le~Merdy theorem, see e.g.\ \cite[Theorem 7.3.7]{Haase06}. In well-behaved examples, this property is often trivially given as the semigroup is even contractive (with respect to a suitably chosen Hilbert space norm), which typically arises from modelling energy balances in physical systems and is usually checked implicitly when well-posedness is argued (Lumer-Phillips theorem). Moreover, for self-adjoint and normal semigroup generators, the contraction property is trivially satisfied. Yet, the question whether a given semigroup is similar to a (Hilbert space) contractive one is, in its general form, not at all trivial, see e.g. \cite{OlT25,OlT25b} for a recent in-depth analysis. In this work, this property will be a structural assumption for some results. Most of our results relate to analytic semigroups on Hilbert spaces relying on classical results relating the boundedness of the holomorphic functional calculus for sectorial operators and similarity to contractions \cite{Haase06}.
}

If $B$ is an unbounded operator, it is well-known that $L^2$-ISS for a linear system is equivalent to $2$-admissibility of $B$ together with exponential stability of the underlying semigroup \cite[Proposition 2.10]{JNP18}. 
In Section~\ref{sec:Non-coercive quadratic ISS Lyapunov functions for analytic systems}, we prove for  semigroups which are analytic and similar to a contraction semigroup that the condition $B \in L(U,X_{-p})$, $p\in (0,\frac{1}{2})$ (which is stronger than 2-admissibility) implies the existence of a coercive $L^2$-ISS Lyapunov function. If $A$ is additionally self-adjoint, then a weaker condition $B \in L(U,X_{-\frac{1}{2}})$  (still slightly stronger than 2-admissibility) already suffices for the same claim.

Section~\ref{sec:Non-coercive quadratic ISS Lyapunov functions for analytic systems} further demonstrates that non-coercive ISS Lyapunov functions can be constructed for linear systems under much less restrictive assumptions. 
\fsc{We further argue the applicability of our abstract results by two examples in Section 5.}
Our findings are summarized in Figure~\ref{fig:Overview-of-results}. 

Our results can be understood as a part of the effort on the development of the methods to analyze ISS of nonlinear PDEs with boundary inputs, or more generally, of nonlinear boundary control systems. 
As a rule, Lyapunov methods seem to be the most realistic way to study stability and ISS of nonlinear systems. Hence, our long-term aim is to rigorously settle  the applicability of Lyapunov methods to linear systems with bounded input operators and then to extend the methods to treat nonlinear PDEs with boundary controls and nonlinear boundary control systems in general.

\textbf{Notation.} Throughout this note, $X$ and $Y$ will refer to Banach spaces which may, at instances, be specified to be a Hilbert space with an inner product $\langle\cdot,\cdot\rangle$. 
Denote $\Rpn:=[0,+\infty)$.

For a Banach space $(U,\|\cdot\|_U)$, we denote by $L^\infty(0,t;U)$ the space of Bochner measurable functions $u:(0,t)\to U$ with finite essential supremum norm 
$\|u\|_{L^\infty(0,t)}:=\esssup_{s\in(0,t)}\|u(s)\|_U$ and similarly we define the common Lebesgue Bochner spaces $L^p(\Rpn,U)$, $p\in [1,\infty]$.
The space of bounded operators acting from $X$ to a Banach space $Y$ we denote by $L(X,Y)$ and  $L(X):=L(X,X)$.\newline
Recall the following well-known classes of comparison functions:
\begin{align*} 
\arraycolsep=1.3pt\def\arraystretch{1.2}
\calK :={}& \{\mu\in C(\Rpn,\Rpn) \:|\: \mu(0)=0, \mu \text{ strictly increasing}\},\\
 \calK_\infty :={}& \{\theta\in\calK \:|\:  \lim_{x\to\infty} \theta(x)=\infty\},\\
\mathcal{L} :={}&\{\gamma\in C(\Rpn,\Rpn)\:|\:\gamma \text{ strictly\ decreasing,}  \lim_{t\to\infty}\gamma(t)=0 \},\\
\mathcal{KL} := {}&\{ \beta:\Rpn\tm\Rpn\rightarrow\Rpn\: | \: \beta(\cdot,t)\in\calK\ \forall t,\quad \beta(s,\cdot)\in\mathcal{L}\ \forall s\neq0\}.
\end{align*}
As is common, the symbol $\lesssim$ is used to drop absolute multiplicative constants that do not depend on the variables appearing in the inequality.

\section{Linear systems and their stability}

In the following, let $A:D(A) \subset X\to X$ always be the infinitesimal generator of a strongly continuous semigroup $T:=(T(t))_{t\geq 0}$ on $X$ with a nonempty resolvent set $\rho(A)$.
Recall that a semigroup $T$ is called analytic if $T$ extends to an analytic mapping $z\mapsto T(z)$ on a sector $S_{\zeta}=\{z\in\mathbb{C}\setminus\{0\}\colon \arg(z)<\zeta\}$ for some $\zeta\in(0,\pi/2]$ and $\lim_{z\to0, z\in S_{\theta}}T(z)x=x$ for all $x\in X$ and some $\theta\in (0,\zeta)$.

For the rest of the paper, we will be interested in systems $\Sigma(A,B)$ given by abstract Cauchy problems of the form 
\begin{equation}
\label{eq:Linear_System}
\dot{x}(t)=Ax(t)+Bu(t),\qquad t>0, \ x(0)=x_{0},
\end{equation}
where 
$B$ is an operator, which is possibly unbounded, acting on the input space $U$. 
The reason for allowing unbounded operators $B$ stems from the study of systems with boundary or point controls \cite{JaZ12,TuW09}. In contrast to the (in general unbounded) operator $A$, $B$ will always be defined on the ``full space'' $U$, and the ``unboundedness'' is only reflected in the norm on the image space. 

To clarify the precise assumptions on $B$, let us recall a solution concept for \eqref{eq:Linear_System}. Consider the function $x$ (formally) given by
\begin{equation}\label{eq:mild}
 x(t)=T(t)x_{0}+\int_{0}^{t}T(t-s)Bu(s)\mathrm{d}s,\qquad t\ge0,
\end{equation}
for any $x_{0}\in X$ and $u\in L_{\mathrm{loc}}^{1}(\R_{\ge0},U)$. 
If  $x$, which we denote also $\phi(\cdot,x_0,u)$, maps $[0,\infty)$ to $X$, then we call $x$ the \emph{mild solution} of \eqref{eq:Linear_System}. 
If $B$ is a bounded operator in the sense that $B\in L(U,X)$, then \eqref{eq:mild} indeed defines such mild solution. 
For more general operators $B$, suitable properties are required, in particular such that the integral in \eqref{eq:mild} is well-defined in $X$ for all $t>0$ and inputs $u$ from a space of $U$-valued (equivalence classes of) functions such as $L^{q}(0,\infty;U)$.

  To introduce these properties, we will view the convergence of the integral in a weaker norm on $X$ as follows. Define the extrapolation space $X_{-1}$ as the completion of $X$ with respect to the norm 
$ \|x\|_{X_{-1}}:= \|(aI -A)^{-1}x\|_X$ for some $a \in \rho(A)$.
$X_{-1}$ is a Banach space (see \cite[Theorem 5.5, p. 126]{EnN00}) and different choices of $a \in \rho(A)$ generate equivalent norms on $X_{-1}$, see \cite[p. 129]{EnN00}.
As we know from the representation theorem \cite[Theorem 3.9]{Wei89b}, the input operator $B$ must satisfy the condition $B\in L(U,X_{-1})$ in order to give rise to a well-defined control system. 
Lifting of the state space $X$ to a larger space $X_{-1}$ is natural because the semigroup  $(T(t))_{t\ge 0}$ extends uniquely to a strongly continuous semigroup  $(T_{-1}(t))_{t\ge 0}$ on $X_{-1}$ whose generator $A_{-1}:X_{-1}\to X_{-1}$ is an extension of $A$ with $D(A_{-1}) = X$, see, e.g.,\ \cite[Section II.5]{EnN00}. If clear from the context, we may drop the subscript ``$-1$'' in our notation. 
Thus we may consider Equation \eqref{eq:Linear_System} on the Banach space $X_{-1}$ by replacing $A$ by $A_{-1}$ and henceforth interpret \eqref{eq:mild} in $X_{-1}$ as the integral exists in $X_{-1}$ when the extension of the semigroup is considered. The standing assumption for systems $\Sigma(A,B)$ is thus that $B\in L(U,X_{-1})$, where $X$ and $U$ are general Banach spaces. The lifting, however, comes at a price: $x$ has values in $X_{-1}$ in general. 
This motivates the following classical definition: 
\begin{definition}[\amc{$q$-admissible control operator}]
\label{def:q-admissibility} 
The operator $B\in L(U,X_{-1})$ is called a {\em $q$-admissible control operator} for $(T(t))_{t\ge 0}$, where $1\le q\le \infty$, if for all $t\geq 0$ and $u\in L^q([0,t],U)$, it holds that
\begin{eqnarray}
 \int_0^t T_{-1}(t-s)Bu(s)ds\in X.
\label{eq:q-admissibility}
\end{eqnarray}
If the analogous property holds for $U$-valued continuous functions $u\in C([0,t],U)$, or regulated functions $\mathrm{Reg}([0,t],U)$, we say that $B$ is $C$-admissible or $\mathrm{Reg}$-admissibile respectively. 
\end{definition}

If $B$ is $q$-admissible, then $x$ defined by \eqref{eq:mild} is indeed a mild solution of \eqref{eq:Linear_System}.
In the language of systems theory, $q$-admissibility of $B$ with respect to $A$ means precisely the forward-completeness of $\Sigma(A,B)$ for all inputs from $L^q$.
Any mild solution is continuous if $B$ is $q$-admissible for $q<\infty$, see \cite[Proposition 2.3]{Wei89b}. In the critical case $q=\infty$, this is also known in many (practically relevant) situations, see e.g.\ \cite{JNP18,JSZ17}, but is an open question in the general case.

For linear systems with  admissible $B$, we study the following stability notions.
\begin{definition}[\amc{$L^p$-input-to-state stability}]
System $\Sigma(A,B)$ is called {\em $L^p$-input-to-state stable ($L^p$-ISS)}, if there exist functions $\beta\in \mathcal{KL}$ and $ \mu\in {\mathcal K}_\infty$ such that for every $x_0\in X$, every $t\ge 0$ and every $u\in L^p(0,t;U)$,
 the mild solution $x$ of \eqref{eq:Linear_System} satisfies $x(t)\in X$ and 
\begin{equation}\label{eqnISSCb}
\left\| x(t)\right\| \le \beta(\|x_0\|,t)+  \mu (\|u\|_{L^p(0,t)}).
\end{equation}
\end{definition}

ISS of \eqref{eq:Linear_System} can be characterized as follows, see \cite{JNP18}.
\begin{proposition}[{\cite[Proposition 2.10]{JNP18}}]
\label{thm:ISS-Criterion-lin-sys-with-unbounded-operators}
Let $p\in[1,+\infty]$. 
The following assertions are equivalent:  
\begin{enumerate}
    \item[(i)] $\Sigma(A,B)$ is $L^p$-ISS.
    \item[(ii)] $A$ generates an exponentially stable semigroup  and $B$ is $p$-admissible.
\end{enumerate}
\end{proposition}

\begin{definition}[ISS Lyapunov function]
\label{def:ISS Lyapunov function}
Consider $\Sigma(A,B)$ and suppose that $B$ is $\infty$-admissible.
A continuous function $V:X\to \Rpn$ is called {a {\em (non-coercive) ISS Lyapunov function}} for $\Sigma(A,B)$ if there exist 
$\alpha_{2},\alpha_{3}\in \mathcal{K_\infty}$ and $\sigma\in \mathcal{K}_{\infty}$ such that 
\begin{align}
\label{eq:one-sided-sandwich}
0 < V(x)\leq{}& \alpha_{2}(\|x\|),\quad  x\neq 0,
\end{align}
and for all $x\in X$ and all $u\in L_{loc}^{\infty}(0,\infty;U)$,
\begin{align}
\dot{V}_{u}(x)\leq{}& - \alpha_{3}(\|x\|) + \sigma(\limsup_{t\to 0^{+}}\|u\|_{L^{\infty}(0,t)}),\label{eq:ISSLF2}
\end{align}
where $\dot{V}_{u}(x)$ is the right-hand Dini derivative of $V(x(\cdot))$ at $t=0$:
\[
\dot{V}_{u}(x):=\limsup_{h\to0^{+}}\frac{1}{h}\Big(V\big(\phi(h,x,u)\big)-V(x)\Big),
\]
and $x(\cdot)$ is the mild solution \eqref{eq:mild} of \eqref{eq:Linear_System} with initial condition $x$ and input $u$. 

An ISS Lyapunov function is called {\em coercive} if there exists  $\alpha_{1}\in\mathcal{K}_\infty$ such that 
 \begin{equation*}
 \alpha_{1}(\|x\|)\leq V(x), \qquad x\in X.
 \end{equation*}
The function $V$ is called {\em Lyapunov function} {(for the uncontrolled system \eqref{eq:Linear_System})} if $V$ is an ISS Lyapunov function for $\Sigma(A,0)$. 
\end{definition}

\begin{definition}[Quadratic \amc{ISS} Lyapunov function] 
\label{def:Quadratic-LF}
Let $X$ be a Hilbert space. 
An ISS Lyapunov function $V:X\to \Rpn$ is called a \emph{quadratic ISS Lyapunov function}  if there exists a self-adjoint operator $P\in L(X)$ such that $V(x)=\langle Px,x\rangle$ for all $x\in X$. In this case, we also say that \emph{$V$ is quadratic}. 
\end{definition}
We call a bounded, self-adjoint operator $P$  \emph{positive} if $\langle Px,x\rangle > 0$ for all $x\in X\setminus\{0\}$.
Clearly, if $V: x \mapsto \langle Px,x\rangle$ is a quadratic Lyapunov function, then $P$ is positive. Furthermore, $P$ is invertible (with a bounded inverse) if and only if $V$ is coercive. 


Classical constructions of Lyapunov functions via solution of the Lyapunov operator equation, see, e.g.,~\cite[Theorem 4.1.3]{CuZ20} and the references therein, lead to quadratic Lyapunov functions. On the other hand, quadratic Lyapunov functions can be easily differentiated, and there are efficient numerical schemes for the construction of quadratic Lyapunov functions, such as the sum of squares (SoS) method.

The following characterization is elementary, but it motivates a definition of a quadratic Lyapunov function for general Banach spaces.
\begin{proposition}
\label{prop:Characterization-quadratic-LFs} 
Let $X$ be a Hilbert space. An ISS Lyapunov function $V:X\to \Rpn$ is a quadratic ISS Lyapunov function if and only if there exists $F\in L(X)$ such that  
\begin{equation}
\label{eq:quadratic-Banach}
V(x)=\|Fx\|^2,\qquad\text{ for all }x\in X.
\end{equation}
\end{proposition}

\begin{proof}
Let $V$ be a quadratic ISS Lyapunov function, i.e.,\ $V(x)=\langle Px,x\rangle$, for all $x\in X$, and for some self-adjoint, positive $P$.
Since $P^{\frac{1}{2}}$ is well-defined, self-adjoint and positive, see, e.g., \cite[Theorem 12.3.4]{TuW09}, we conclude that $V(x)=\|P^{\frac{1}{2}}x\|^2$ for all $x\in X$. 
Conversely, let $V$ be an ISS Lyapunov function such that there is $F\in L(X)$ with  $V(x)=\|Fx\|^2$ for all $x\in X$. 
Then 
\[
0< V(x) = \langle Fx,Fx\rangle = \langle F^*Fx,x\rangle,\quad x\neq 0,
\]
and thus $P=F^*F$ is self-adjoint and positive, and hence $V$ is quadratic.
\end{proof}

In view of Proposition~\ref{prop:Characterization-quadratic-LFs}, in the Banach space setting, we will call functions $V$ as in \eqref{eq:quadratic-Banach} quadratic ISS Lyapunov functions.

The use of $\mathcal{K}$-functions in the definition of ISS Lyapunov functions is natural in the view of nonlinear systems. 
It is not surprising that for quadratic Lyapunov functions there is no need to consider general comparison functions.

\begin{lemma}
\label{lem:quadratic}
Let $X$ be a Hilbert space. Let $V:X\to \Rpn$ be a coercive, quadratic ISS Lyapunov function for $\Sigma(A,B)$ with $\infty$-admissible $B$ and with $\sigma(r)=a_4r^2$ in \eqref{eq:ISSLF2} for some $a_{4}>0$. Then there exist constants $a_{1},a_{2},a_{3}>0$ such that 
\begin{equation}
\label{eq:quadratic-sandwich}
a_{1}\|x\|^{2}\leq V(x)\leq a_{2}\|x\|^{2},\qquad x\in X,
\end{equation}
and for all $u\in L^\infty(\Rpn,U)$ we have
\begin{equation}
\label{eq:quadratic-decay}
\dot{V}_{u}(x) \leq -a_3 \|x\|^{2} + a_4 (\limsup_{t\to 0^{+}}\|u\|_{L^{\infty}(0,t)})^2,\quad x\in X.
\end{equation}
\end{lemma}

\begin{proof}
Let $V$ be a quadratic ISS Lyapunov function for $\Sigma(A,B)$ with $\alpha_1,\alpha_2$ as in Definition~\ref{def:ISS Lyapunov function}.
Then $V(x) = V(\frac{x}{\|x\|})\|x\|^2$ for $x\neq 0$ and 
\[
a_1:=\alpha_1(1) \leq V\Big(\frac{x}{\|x\|}\Big)\leq \alpha_2(1)=:a_2.
\]
For $x \in X$, $x\neq 0$, and any $u\in L^\infty(\Rpn,U)$, we compute using joint linearity of the flow 
$\phi$ with respect to $x$, and $u$, that
\begin{eqnarray*}
\dot{V}_{u}(x) 
& = & \limsup_{t\to 0^{+}}\frac{V(\phi(t,x,u)) - V(x)}{t}\\
& = & \limsup_{t\to 0^{+}}\frac{V\Big(\phi\big(t,\|x\|\frac{x}{\|x\|},\|x\|\frac{u}{\|x\|}\big)\Big) - V\big(\|x\|\frac{x}{\|x\|}\big)}{t}\\
& = & \|x\|^2\limsup_{t\to 0^{+}}\frac{ V\Big(\phi\big(t,\frac{x}{\|x\|},\frac{u}{\|x\|}\big)\Big) - V\big(\frac{x}{\|x\|}\big)}{t}\\
& \leq & \|x\|^2\Big(-\alpha_3(1) + \frac{a_4}{\|x\|^2} \big(\limsup_{t\to 0^{+}}\|u\|_{L^{\infty}(0,t)}\big)^2 \Big)\\
& = & -\alpha_3(1)\|x\|^2 + a_4 \big(\limsup_{t\to 0^{+}}\|u\|_{L^{\infty}(0,t)}\big)^2.
\end{eqnarray*}
\end{proof}

ISS Lyapunov functions as defined above are of virtue to study $L^\infty$-ISS. For the analysis of $L^p$-ISS, another type of ISS Lyapunov functions is needed.

\begin{definition}[\amc{$L^p$-ISS Lyapunov function}]
\label{def:L2-ISS Lyapunov function}\-
Let $p\in[1,\infty)$ and let $\Sigma(A,B)$ be a system with $C$-admissible $B\in L(U,X_{-1})$. A continuous function $V:X\to \Rpn$  is called {a {\em (non-coercive) $L^p$-ISS Lyapunov function}} for $\Sigma(A,B)$  if there is $\alpha_2\in \mathcal{K}_{\infty}$ such that \eqref{eq:one-sided-sandwich} holds, and there exist constants $a_{3},a_4>0$ such that 
\begin{equation}
\label{eq:quadratic-decay-L2-ISS}
\dot{V}_{u}(x) \leq -a_3 \|x\|^{2} + a_4\|u(0)\|^p_U,
\end{equation}
 for all $x \in X$ and all $u\in C(\Rpn,U)$. 
If there exists an (injective) operator $F \in L(X)$ such that \eqref{eq:quadratic-Banach} holds, we say that $V$ is a \emph{quadratic $L^p$-ISS Lyapunov function}, which is called \emph{coercive} if additionally \eqref{eq:quadratic-sandwich} is satisfied for some $a_1,a_2>0$.
\end{definition}
We emphasize the difference between the notions of \emph{ISS Lyapunov function} and \emph{$L^{p}$-ISS Lyapunov function}.

\begin{proposition}
\label{prop:L2-ISS-Lyapunov-suff-cond} 
If $B$ is $C$-admissible and there is a coercive quadratic $L^2$-ISS Lyapunov function for  $\Sigma(A,B)$, then  $\Sigma(A,B)$ is $L^p$-ISS for all $p\in[2,+\infty]$.
\end{proposition}

\begin{proof}
As the flow $\phi$ of $\Sigma(A,B)$ depends continuously on inputs, the claim follows from \cite[Theorem 1]{Mir20}, where, using nonlinear rescaling of $V$, an explicit construction of the (non-quadratic) $L^q$-ISS Lyapunov functions was provided, for all $q\in(2,+\infty)$. The case $p=\infty$ directly follows from Proposition \ref{thm:ISS-Criterion-lin-sys-with-unbounded-operators}.
\end{proof}

Next, we show that $L^p$-ISS Lyapunov functions for linear systems cannot be quadratic unless $p=2$.
\begin{proposition}
\label{prop:Restrictions-quadratic-LFs} 
Let $V$ be a (coercive or non-coercive) quadratic $L^p$-ISS Lyapunov function for $\Sigma(A,B)$ with $B\neq 0$ and $p\in[1,+\infty)$. Then $p=2$.
\end{proposition}

\begin{proof}
By definition, there exists $a \geq 0$ such that
\[
\dot{V}_{u}(0) \leq a\|u(0)\|^p_U,\quad u\in C(\Rpn,U).
\]
Take $a_m$ as the infimum of $a>0$ satisfying the previous property. 
As $B\neq 0$, $a_m>0$. By the linearity of $(t,u) \mapsto \phi(t,0,u)$ in $u$, and as $V$ is quadratic, we see that for any $u\in C(\Rpn,U)$, and any $c>0$
\begin{align*}
\dot{V}_{cu}(0) 
&= \limsup_{h\to +0}\frac{V(\phi(h,0,cu))}{h} = c^2\limsup_{h\to +0}\frac{V(\phi(h,0,u))}{h}\\
&\leq c^2 a_m\|u(0)\|^p_U = c^{2-p} a_m\|cu(0)\|^p_U,
\end{align*}
and taking $c>1$ for $p>2$ and $c<1$ for $p<2$, we come to a contradiction to the choice of $a_m$.
\end{proof}

As a corollary of Lemma~\ref{lem:quadratic}, we obtain 
\begin{proposition}
\label{prop:Linfty-ISS-Lyapunov-fun-and-L2-ISS-Lf}
Consider a system $\Sigma(A,B)$ with $C$-admissible $B$.
Let $V$ be a quadratic ISS Lyapunov function for $\Sigma(A,B)$ with $\sigma(r)=ar^2$ for some $a>0$. 
Then $V$ is an $L^2$-ISS Lyapunov function.
\end{proposition}

\amc{
\begin{remark}
Let $A$ generate an exponentially stable semigroup $T$ on a Hilbert space $X$.
\begin{itemize}
 \item[(i)] If $B$ is $C$-admissible, but not 2-admissible, then there is no coercive quadratic $L^p$-ISS Lyapunov function for $p\in[1,+\infty)$. 
As if such a function would exist, then $p=2$ by Proposition~\ref{prop:Restrictions-quadratic-LFs}, $\Sigma(A,B)$ would be $L^2$-ISS by Proposition~\ref{prop:L2-ISS-Lyapunov-suff-cond}, which would imply that $B$ is 2-admissible, a contradiction. 

 \item[(ii)] At the same time, under certain (verifiable) assumptions on $A$, \cite[Theorem 5.3]{JMP20} ensures for $\infty$-admissible $B$ the existence of a non-coercive quadratic ISS Lyapunov function $V$ for $\Sigma(A,B)$ with $\sigma(r)=ar^2$ for some $a>0$. 
Proposition~\ref{prop:Linfty-ISS-Lyapunov-fun-and-L2-ISS-Lf} shows that $V$ is an $L^2$-ISS Lyapunov function. It has to be non-coercive by the above argument. 
In particular, this means that $C$-admissibility of $B$ together with the existence of a non-coercive $L^2$-ISS Lyapunov function does not ensure $L^2$-ISS, i.e. coercivity in Proposition~\ref{prop:L2-ISS-Lyapunov-suff-cond} cannot be relaxed to non-coercivity. This indicates that non-coercive Lyapunov theory is well-suited for analysis of systems with continuous or $L^\infty$ inputs, but is not suited for analysis of $L^p$-ISS with $p\in[1,+\infty)$.
\end{itemize}

\end{remark}
}

\section{Upgrading quadratic Lyapunov functions to $L^2$-ISS Lyapunov functions}

In the sequel, it will be of interest for us to determine under which conditions a quadratic Lyapunov function for the undisturbed system $\Sigma(A,0)$ is also an $L^2$-ISS Lyapunov function for the system $\Sigma(A,B)$. 
Here, we show several results of this kind.

\begin{proposition}
\label{prop:quadratic-ISS-LFs-with-inputs-Banach} 
Let $X$ be a Banach space and let $V:X \to \Rpn$ be a quadratic (coercive or non-coercive) Lyapunov function for $\Sigma(A,0)$ with $V(x) = \|Fx\|^2$ for a certain $F\in L(X)$ and all $x\in X$. Assume that $B \in L(U,X_{-1})$ is C-admissible and that there exists $K>0$ such that for each $u \in C(\Rpn,U)$
\begin{eqnarray}
\limsup_{t\to +0}\frac{1}{t}\Big\|F\int_{0}^{t}T_{-1}(t-s)Bu(s)\mathrm{d}s\Big\| \leq C \|u(0)\|_{U}.
\label{eq:Requirement-on-B}
\end{eqnarray}
Then $V$ is a quadratic (coercive or non-coercive, respectively) $L^2$-ISS Lyapunov function for $\Sigma(A,B)$. In particular, \eqref{eq:Requirement-on-B} holds for $B\in L(U,X)$.
\end{proposition}

\begin{proof}
For any $x \in X$ and $u\in C(\Rpn,U)$ we have by triangle inequality
\begin{align}
\frac{1}{t}\Big[\Big\|F\Big(T(t)x+ & \int_{0}^{t}T_{-1}(t-s)Bu(s)\mathrm{d}s\Big)\Big\|^2 - \|Fx\|^2 \Big] \nonumber\\
&\leq   \frac{1}{t}\Big[\|FT(t)x\|^2 - \|Fx\|^2\Big] + \frac{1}{t^2}\cdot t\Big\|F\int_{0}^{t}T_{-1}(t-s)Bu(s)\mathrm{d}s\Big\|^2 \nonumber\\
&+ \frac{2}{t}\|FT(t)x\| \Big\|F\int_{0}^{t}T_{-1}(t-s)Bu(s)\mathrm{d}s\Big\|.
\label{eq:aux-quadratic-ISS-LF}
\end{align}
Taking the limit $t\to+0$, using that
\[
\limsup_{h\to0^{+}} \frac{\|FT(t)x\|^2 - \|Fx\|^2}{t} 
= \dot{V}_{u=0}(x),
\]
and exploiting the property \eqref{eq:Requirement-on-B}, we obtain that 
\begin{align*}
\dot{V}_u(x) &=\limsup_{t\to0{+}}\frac{1}{t}\Big[\Big\|F\Big(T(t)x+\int_{0}^{t}T_{-1}(t-s)Bu(s)\mathrm{d}s\Big)\Big\|^2 - \|Fx\|^2 \Big]\\
&\leq  \dot{V}_{u=0}(x) + 2\|Fx\| K \|u(0)\|_{U}.
\end{align*}
In view of Lemma~\ref{lem:quadratic},  the decay rate of quadratic Lyapunov functions is quadratic, whence there is $a_3>0$ such that 
\begin{align*}
\dot{V}_u(x) 
&\leq  -a_3 \|x\|^2 + 2 \|x\| \|F\| K \|u(0)\|_{U}.
\end{align*}
Using Young's inequality, we obtain for any $\varepsilon>0$ that 
\begin{align*}
\dot{V}_u(x) 
&\leq  -a_3 \|x\|^2 + \varepsilon \|x\|^2  + \frac{1}{\varepsilon} \|F\|^2 K^2\|u(0)\|_{U}^2,
\end{align*}
and choosing $\varepsilon \in (0,a_3)$, we see that $V$ is a quadratic $L^2$-ISS Lyapunov function for $\Sigma(A,B)$.

If $B\in L(U,X)$, then \eqref{eq:Requirement-on-B} follows since for any $u\in C(\Rpn,U)$,
\begin{eqnarray*}
\lim\limits_{h \rightarrow +0} \frac{1}{h} \int_0^h{T(h-r)B u(r)dr} = B u(0), 
\end{eqnarray*}
\end{proof}

We note that condition \eqref{eq:Requirement-on-B} nearly implies boundedness of $B$ if the Lyapunov function is coercive, i.e.\ $F$ is boundedly invertible. Indeed, in that case it can be  easily shown that $B$ is $L^{1}$-admissible, which is known to imply that $B\in L(U,X)$ if $X$ is reflexive, \cite{Wei89b}.

\section{Coercive quadratic Lyapunov functions for linear systems with bounded input operators}
\label{sec:Coercive quadratic Lyapunov functions for linear systems with bounded input operators}

Lyapunov functions are important since they are certificates for stability properties. As we are interested in ISS in this note, we recall the following result from \cite[Theorem 8]{MiW17c}:
\begin{proposition}
\label{prop:Lyapunov-criterion-for-ISS-bounded-operators} 
Let $X$ be a Hilbert space, $A$ be the generator of a strongly continuous semigroup and $B\in L(U,X)$.
The following statements are equivalent:
\begin{itemize}
 \item[(i)] \eqref{eq:Linear_System} is $L^p$-ISS for some $p\in [1,+\infty]$.
 \item[(ii)] \eqref{eq:Linear_System} is $L^p$-ISS for all $p\in [1,+\infty]$.
 \item[(iii)] There is a coercive $L^1$-ISS Lyapunov function for \eqref{eq:Linear_System}, which is an equivalent norm on $X$.
 \item[(iv)] There is a non-coercive quadratic $L^2$-ISS Lyapunov function for \eqref{eq:Linear_System}.
\end{itemize}
\end{proposition}
Coercive $L^1$-ISS Lyapunov functions constructed in the proof of Proposition~\ref{prop:Lyapunov-criterion-for-ISS-bounded-operators} in \cite{MiW17c} to show the equivalence between (i) and (iii) are never quadratic. In fact, they are norms on $X$, equivalent to 
$\|\cdot\|$, and thus they are homogeneous of degree one. 
In this section, we show a criterion for the existence of a quadratic coercive $L^2$-ISS Lyapunov function for linear systems with bounded input operators.

We say that a \emph{semigroup $T$ is similar to a contraction semigroup}, if there exists a boundedly invertible operator $S:X\to X$ so that $(ST(t)S^{-1})_{t\geq 0}$ is a contraction semigroup. 

 Furthermore, we call $\langle \cdot,\cdot\rangle_{new}$ an \emph{equivalent scalar product} in $X$, if $\langle \cdot,\cdot\rangle_{new}$ is a scalar product in $X$ that induces a norm $\|\cdot\|_{new} = \sqrt{\langle \cdot,\cdot\rangle_{new}}$, that is equivalent to the norm $\|\cdot\|$ in $X$. The following lemma is well-known, and the proof is elementary using the Lumer--Phillips theorem.

\begin{lemma}
\label{lem:Similar-to-contraction-andequivalent-scalp} 
Let $X$ be a Hilbert space, and $A$ be the generator of a $C_0$-semigroup $T$.
Then, $T$ is similar to a contraction semigroup if and only if there exists an equivalent scalar product $\langle \cdot,\cdot\rangle_{new}$ in $X$ such that $A$ is dissipative, i.e.\ 
\begin{eqnarray}
\Re\langle Ax,x\rangle_{new}\leq 0\quad \forall x\in D(A).
\label{eq:dissipativity}
\end{eqnarray}
\end{lemma}

The following result settles the existence of coercive, quadratic Lyapunov functions for systems with bounded $B$.
\begin{theorem}
\label{thm:Quad-coercive-nonanalytic-1} 
Let $X$ be a Hilbert space, $B\in L(U,X)$, and let $A$ generate an exponentially stable semigroup $T$ on $X$. The following statements are equivalent:
\begin{enumerate}[label=(\roman*)]
\item\label{thm:quad-coercivity-itm0} There exists a coercive, quadratic $L^2$-ISS Lyapunov function for \eqref{eq:Linear_System}.
\item\label{thm:quad-coercivity-itm1} There exists a coercive, quadratic Lyapunov function for \eqref{eq:Linear_System} with $B=0$.
\item\label{thm:quad-coercivity-itm20} There exists an equivalent scalar product $\langle \cdot,\cdot\rangle_{new}$ such that $A$ is dissipative, i.e.
\begin{eqnarray}
\Re\langle Ax,x\rangle_{new}\leq 0,\quad x\in D(A)
\label{eq:Dissipativity-equivalent-scalp}
\end{eqnarray} 
and $x\mapsto \scalp{x}{x}_{new} = \|x\|^2_{new}$ is a quadratic ISS Lyapunov function for \eqref{eq:Linear_System}.

\item\label{thm:quad-coercivity-itm2} There exists an equivalent scalar product $\langle \cdot,\cdot\rangle_{new}$ such that $A$ is dissipative.

\item\label{thm:quad-coercivity-itm3} $T$ is similar to a contraction semigroup.

\item\label{thm:quad-coercivity-itm32} There is an equivalent norm in $X$ of the form $W(x) = \|Fx\|$, for $F \in L(X)$, such that $\dot{W}(x) \leq -W(x)$ for $B=0$.
\end{enumerate}
\end{theorem}

\begin{proof}
\ref{thm:quad-coercivity-itm0} $\Iff$ \ref{thm:quad-coercivity-itm1}. Follows by Proposition~\ref{prop:quadratic-ISS-LFs-with-inputs-Banach}.

\ref{thm:quad-coercivity-itm1} $\Rightarrow$ \ref{thm:quad-coercivity-itm20}. 
Assume that there exists a coercive, quadratic Lyapunov function $V:X\to \Rpn$. By Lemma \ref{lem:quadratic}, there exists a self-adjoint operator $P\in L(X)$ and constants $c_{1},c_{2}>0$ such that $V(x)=\langle Px,x\rangle$  and
\begin{equation}
\label{eq:Bounds-for-V-aux}
c_{1}\|x\|^{2}\leq \langle Px,x\rangle\leq c_{2}\|x\|^{2},\qquad x\in X.
\end{equation}
Thus, $\langle x,y\rangle_{new}:=\langle P^{\frac{1}{2}}x,P^{\frac{1}{2}}y\rangle$ defines a new scalar product with a norm $\|\cdot\|_{new}$ which is equivalent to $\|\cdot\|$. For any $x\in D(A)$, 
\begin{align*}
\frac{1}{h}\Big(V(T(h)x)-&V(x)\Big)={}\frac{1}{h}\Big(\langle T(h)x,Px\rangle-\langle x,Px\rangle\Big)\\{}&\to \langle Ax,Px\rangle = \langle Ax,x\rangle_{new} 
\ \text{ as } h\to 0^{+}.
\end{align*}
By the assumption that $V$ is a Lyapunov function, we conclude that $A$ is dissipative with respect to the new scalar product.

\ref{thm:quad-coercivity-itm20} $\Rightarrow$ \ref{thm:quad-coercivity-itm2}. Clear. 

\ref{thm:quad-coercivity-itm2} $\Rightarrow$ \ref{thm:quad-coercivity-itm1}. 
As $T$ is exponentially stable, and $\langle \cdot,\cdot\rangle_{new}$  is an equivalent scalar product, 
$T$ is also exponentially stable in the new norm $\|\cdot\|_{new}$ and thus there exists a positive operator 
$\tilde{P}\in L(X)$ satisfying the inequality 
\begin{eqnarray}
\Re\langle Ax,Px\rangle_{new} \leq -\langle x,x \rangle_{new},\qquad \forall x\in D(A)
\label{eq:Lyapunov-ineq-new-scalp}
\end{eqnarray}
with $P:=\tilde{P}$. In view of \eqref{eq:Dissipativity-equivalent-scalp}, for every $\varepsilon>0$ the operator $P=\tilde{P}+\varepsilon I$ also solves \eqref{eq:Lyapunov-ineq-new-scalp}.
Obviously, $V:x\mapsto \langle Px,x\rangle$,  $x\in X$ is a coercive Lyapunov function.

\ref{thm:quad-coercivity-itm2} $\Iff$ \ref{thm:quad-coercivity-itm3}. Follows by Lemma~\ref{lem:Similar-to-contraction-andequivalent-scalp}. 

\ref{thm:quad-coercivity-itm1} $\Iff$ \ref{thm:quad-coercivity-itm32}.
 Let $V(x) = \|Fx\|^2$ be a coercive quadratic Lyapunov function for $B=0$ with a certain $F\in L(X)$. Define $W(x):=\sqrt{V(x)}=\|Fx\|$, $x \in X$.
As $V$ is coercive and quadratic, $W$ is an equivalent norm, and furthermore, for certain $a>0$ and all $x \neq 0$ we have that 
\begin{eqnarray*}
\dot{W}(x) = \frac{1}{2 W(x)} \dot{V}(x) \le  - \frac{1}{2 W(x)} 2a V(x) =   -aW(x).
\end{eqnarray*}
The converse implication is analogous.
\end{proof}

\begin{remark}
\label{rem:Norms-as-LFs} 
By item (iii) of Proposition~\ref{prop:Lyapunov-criterion-for-ISS-bounded-operators}, for exponentially stable semigroups $T$, we can always find an equivalent norm, which is a Lyapunov function for an undisturbed system. However, it is in general not of the form $V(x)=\|Fx\|$ for a linear $F \in L(X)$. To have this additional property, we need to assume that the semigroup $T$ is similar to a contraction semigroup.
\end{remark}

Supported by Theorem~\ref{thm:Quad-coercive-nonanalytic-1}, we have the following negative result on the existence of coercive quadratic Lyapunov functions for exponentially stable systems.

\begin{proposition}
\label{prop1}
For any infinite-dimensional Hilbert space $X$, there exists a generator $A$ of an exponentially stable, analytic semigroup $T$ on $X$ such that the system $\Sigma(A,0)$ has no coercive, quadratic Lyapunov function.
\end{proposition}

\begin{proof}
Assume that there exists a coercive, quadratic Lyapunov function $V:X\to \Rpn$. 
By Theorem~\ref{thm:Quad-coercive-nonanalytic-1}, it follows that $A$ generates a semigroup, similar to a contraction semigroup.
%
However, for any infinite-dimensional Hilbert space, it is possible to construct analytic, exponentially stable semigroups, which are not similar to a contraction semigroup, see \cite[Theorem 1.1]{LeMerdy98}, and \cite{mcintoshHinfNo}, as well as \cite[Chapter 9, Theorem 9.17]{Haase06}. 
\end{proof}

\begin{remark}
\label{rem:Existence of a quadratic ISS LF} 
\begin{enumerate}
\item
Note that if one drops in Proposition~\ref{prop1} the condition that the semigroup generated by $A$ is analytic, then examples of semigroups not similar to a contraction semigroup were known for a while, \cite{Fogu69,Che76}. The fact that there exist such examples even for analytic semigroups is more subtle. In particular, the latter examples are rather pathological in the sense that they do not emerge from  PDE examples.
\item
Combining Theorem~\ref{thm:Quad-coercive-nonanalytic-1} with Proposition~\ref{prop:Lyapunov-criterion-for-ISS-bounded-operators}, we obtain another negative result: \emph{Existence of a coercive ISS Lyapunov function for $\Sigma(A,B)$ does not imply existence of a coercive quadratic ISS Lyapunov function for $\Sigma(A,B)$ (in contrast to the finite-dimensional linear case)}.
\end{enumerate}
\end{remark}


Recall that for a generator $A$ of an exponentially stable analytic semigroup, the operator $-A$ is sectorial (of angle less than $\pi/2$), see \cite{Haase06}. Thus, the fractional power $(-A)^{{-\alpha}}$, $\alpha\in(0,1)$ can be defined as a bounded operator via a contour integral of an operator-valued analytic function; this is an instance of the Riesz--Dunford functional calculus for sectorial operators. It can be shown that this operator is injective since $A$ is injective and hence $(-A)^{\alpha}$ can be defined as the inverse. The operator $(-A)^{\alpha}$ is closed and  densely defined. For more information on this construction, see the books \cite{Haase06}  and  \cite[Section 1.4]{Hen81} and, for a very brief description of the essentials required here, \cite{Sch20}.

The following result provides an alternative construction of a coercive quadratic Lyapunov function for exponentially stable \emph{analytic} systems with bounded $B$. It is based on a well-known characterization of analytic semigroups being similar to a contraction semigroup, due to Callier--Grabowski--Le~Merdy  \cite[Theorem 7.3.7]{Haase06}, see also the comments in \cite[Section 7.5]{Haase06}. Recall that if $A$ generates an analytic semigroup $T$, then for any $x\in X$, $\beta:t\mapsto \|AT(t)x\|$ is dominated by a constant times $t^{-1}$.  However, if the semigroup is exponentially stable and $x\in D(A)$, then $\beta$ is integrable. 

\begin{theorem}
\label{thm1} 
Let $X$ be a Hilbert space, $B\in L(U,X)$, and let $A$ generate an exponentially stable analytic semigroup $T$ on $X$. The conditions \ref{thm:quad-coercivity-itm0}--\ref{thm:quad-coercivity-itm32} in Theorem~\ref{thm:Quad-coercive-nonanalytic-1} are equivalent to 
\begin{enumerate}[label=(\roman*), start = 7]
\item\label{thm:quad-coercivity-itm4} The function
\begin{equation}
\label{eq:LFV}
V:D(A)\to \Rpn,\quad  x\mapsto \int_{0}^{\infty}\|(-A)^{\frac{1}{2}}T(t)x\|^{2}\,\mathrm{d}t
\end{equation}
extends to a coercive, quadratic $L^2$-ISS Lyapunov function from $X$ to $\Rpn$ for the system $\Sigma(A,B)$. We denote this extension again by $V$. 
\end{enumerate}
\end{theorem}

%

\begin{proof} 
The implication \ref{thm:quad-coercivity-itm4} $\Rightarrow$ \ref{thm:quad-coercivity-itm1} is clear,  whereas
\ref{thm:quad-coercivity-itm3} $\Rightarrow$ \ref{thm:quad-coercivity-itm4} is a consequence of a more general result for possibly unbounded operators $B$,  Corollary \ref{thm:Coercive-LFs-for-admissible-operators} below.
\end{proof}

\begin{remark}
\label{rem:LF-self-adjoint-case} \begin{enumerate}
\item
The function $V$ defined in Theorem \ref{thm1} is a special instance of more general {\emph{square function estimates} (or quadratic estimates)} originating from harmonic analysis and omnipresent in the holomorphic functional calculus, see \cite{Haase06}. 
\item
The Lyapunov function in \eqref{eq:LFV} takes a particularly simple form in case if $A$ is a self-adjoint operator. 
Using that $A$ and hence $T(t)$  and $(-A)^{\frac{1}{2}}$ are self-adjoint, 
\begin{eqnarray*}
V(x) = \scalp{\int_0^{+\infty}(-A)T(2t)x dt}{x} =-\frac{1}{2}\scalp{\int_0^{+\infty}\frac{d}{dt}(T(t)x) dt}{x}=\frac{1}{2}\scalp{x}{x},  \\
\end{eqnarray*}
where $x\in D(A)$ and we used that the semigroup is exponentially stable. 
\end{enumerate}
\end{remark}

\section{(Non-)coercive quadratic ISS Lyapunov functions for analytic systems on Hilbert spaces}
\label{sec:Non-coercive quadratic ISS Lyapunov functions for analytic systems}

By Proposition~\ref{prop:L2-ISS-Lyapunov-suff-cond} and Proposition~\ref{thm:ISS-Criterion-lin-sys-with-unbounded-operators},
the existence of a quadratic $L^2$-ISS Lyapunov function implies 2-admissibility of $B$, for any system $\Sigma(A,B)$. 

In this section, we establish the converse results for analytic semigroups that are similar to contraction semigroups. 
 We recall the definition of the interpolation space $X_{\alpha}$, $\alpha\in (0,1)$ given through the fractional power $(\lambda-A)^{-\alpha}$, for $\lambda$ in the resolvent set of $A$, acting as an isometric isomorphism between $X$ and {\color{blue}$D((-A)^{\alpha})$ equipped with its graph norm}. Analogously, using the same construction for the extended semigroup on $X_{-1}$ with generator $A_{-1}$, the space $X_{\alpha}$ is defined for $\alpha\in(-1,0)$. For details, see, e.g.,\ \cite{EnN00}.

The following lemma collects some known sufficient conditions for the admissibility of $B$, which will be helpful on this way.
\begin{lemma}
\label{lem:Sufficient-condition-for-admissibility} 
Let $A$ be an analytic semigroup over a Hilbert space $X$ \fsc{, $B\in L(U,X_{-1})$ and $\beta\in[0,1)$.} Then 
\begin{enumerate}[label=(\roman*)]
 \item \fsc{$B \in L(U,X_{-\gamma})$ for all  $\beta<\gamma$ if and only if $B$ is $q$-admissible for all $q \in (\frac{1}{1-\beta},+\infty]$.}

 \item If $A$ is self-adjoint, and $B \in L(U,X_{-\frac{1}{2}})$, then $B$ is 2-admissible.
 \item
  If $A$ is self-adjoint, and $B$ is 2-admissible, then $B$ does \fsc{in general not satisfy} $B \in L(U,X_{-\frac{1}{2}})$.
 \item \amc{If $B$ is $q$-admissible for some $q<2$, then $B\in L(U,X_{-\gamma})$ for some $\gamma<\frac{1}{2}$.}
\end{enumerate}
\end{lemma}

\begin{proof}
Assertion (i) can e.g.\ be found in  \cite[Theorem 1]{Preussler23}, while
 (ii) is immediate from \cite[Proposition 5.1.3]{TuW09} by duality of admissible control and observation operators. 

(iii): The claim follows from the following counterexample, which is a direct adaptation of \cite[Example~5.3.11]{TuW09} by duality. 
Consider the state space $X = \ell_2$ of square summable real sequences, a diagonal operator $A = (-2^n)_{n\in\Z_+}$, and an input operator 
$B:\C\to X$, given by $Bc = (2^{n/2})_{n\in\Z_+}c$.
Then 
\[
(-A)^{-\frac{1}{2}}B = (1,1,1,\ldots),
\]
which is not a well-defined operator from $\C$ to $X$.

At the same time, $A$ is 2-admissible by an application of a Laplace--Carleson measure criterion, \cite[Thm.~3.2]{JPP14}. 
However, $(-A)^{-p}B = (2^{(\frac{1}{2} - p)n})_{n\in\Z_+}$ is a bounded operator from $\C$ to $X$ for $p>\frac{1}{2}$.

\amc{
(iv): If $B$ is $q$-admissible for some $q<2$, then 
$B$ is $r$-admissible for all $r \in (\frac{1}{1-\beta},+\infty]$ with some $\beta <\frac{1}{2}$, and by 
(i), we have that $B\in L(U,X_{-\gamma})$ for some $\gamma<\frac{1}{2}$.

Conversely, if $B\in L(U,X_{-\gamma})$ for some $\gamma<\frac{1}{2}$, then 
$B\in L(U,X_{-s})$ for all $\gamma<s$, and by (i), $B$ is $q$-admissible for some $q<2$.
}
%
\end{proof}

\fsc{We point out that further  conditions for assessing admissibility can  be found in \cite{Haak10,JPP14,Preussler23}.}
Recall the following elementary fact for analytic semigroups $(T(t))_{t\ge0}$ generated by $A$ that are exponentially stable.  As $T(t)$ maps $X$ into $D(A)$ for any $t>0$,  $(-A)^{q}T(t)$ is a well-defined bounded operator on $X$, where $(-A)^{q}$ is defined as the fractional power of sectorial operators.  Furthermore, if $(T(t))_{t\ge0}$ is exponentially stable in addition, then for any $q>0$ there are $M,\delta>0$ such that 
\begin{eqnarray}
\big\|(-A)^{q}T(t)\big\| \leq M t^{-q}e^{-\delta t},\quad t>0.
\label{eq:Fractional-powers-analytic-semigroup}
\end{eqnarray}
For details of these well-known facts, we refer to, e.g.\ \cite{EnN00,Haase06}.

\begin{remark}
We recall the following fact about mild solutions related to analytic semigroups. If $u\in C^{2}(\Rpn,U)$ and $x_{0}\in D(A)$, then the mild solution \eqref{eq:mild}  is classical on $(0,\infty)$ and $x\in C^{1}((0,\infty);X)\cap C([0,\infty);X)$ with $A_{-1}x(s)+Bu(s)\in X$ for all $s>0$. This can be seen by combining \cite[Remark 4.2.9]{TuW09} with the fact that the mild solution can be written as \[x(t)=\tilde{x}(t)+T(t)A^{-1}Bu(0)-A^{-1}Bu(0)\] with $\tilde{x}(t)=T(t)x_{0}+\int_{0}^{t}T(t-s)B[u(s)-u(0)]\mathrm{d}s$ and noting that $T(\cdot)A^{-1}Bu(0)-A^{-1}Bu(0)\in C^{1}((0,\infty);X)\cap C([0,\infty);X)$ by the analyticity. 
\end{remark}

We now obtain a family of non-coercive Lyapunov functions.
This result shows that there is a transition from the non-coercive case ($q=0$) to the coercive case ($q=1/2$ and under the additional assumption that $X$ is a Hilbert space and the semigroup is similar to a contraction). 

\begin{theorem}
\label{thm:Non-Coercive-L2-ISS-LFs} 
Let $X$ be a Hilbert space, and let $A$ generate an exponentially stable analytic semigroup $T$ on $X$.
 \fsc{Furthermore, let $\gamma,q\in[0,1)\times[0,\frac{1}{2})$ with $\gamma+q<1$. Then for every $B\in L(U,X_{-\gamma})$, }
 
\begin{equation}
\label{eq:LFV-non-coercive-special}
W_{q}(x) = \int_{0}^{\infty}\|(-A)^{q}T(t)x\|^{2}\,\mathrm{d}t,\quad x \in X,
\end{equation}
is a non-coercive $L^2$-ISS Lyapunov function for \eqref{eq:Linear_System}. \\
\fsc{If, additionally, $W_{\frac{1}{2}}(x)<\infty$ for all $x\in X$, then, then the assertion also holds for $q=\frac{1}{2}$ and $\gamma<\frac{1}{2}$.}
\end{theorem}

\begin{proof} 
As $T$ is analytic and exponentially stable, by \eqref{eq:Fractional-powers-analytic-semigroup}, the integral in \eqref{eq:LFV-non-coercive-special} converges for $q<\frac{1}{2}$, i.e., $W_{q}$ is well-defined, and 
\[
W_{q}(x)\lesssim \|x\|^{2},\quad x\in X.
\]
\fsc{In case that $W_{\frac{1}{2}}(x)<\infty$ for all $x\in X$, the above inequality for $q=\frac{1}{2}$ follows by a standard closed graph argument. Since $(-A)^{q}$ is invertible, strong continuity of the semigroup implies that $W_q(x) >0$ for $x\neq 0$ and since $(W_{q})^{\frac{1}{2}}$ is a bounded linear operator by the above, $W_q$ is continuous. }

Consider  $q\leq \frac{1}{2}$ and let $x_{0}\in D(A)$ and $u\in C^{2}(\Rpn;U)$. Consider the mild solution \eqref{eq:mild}, which is classical on $(0,\infty)$ and $x\in C^{1}((0,\infty);X)\cap C([0,\infty);X)$ with $A_{-1}x(s)+Bu(s)\in X$ for all $s>0$.  Then for   $s>0$
 \begin{align}
  \frac{\mathrm{d}}{\mathrm{d}s} W_{q}(x(s)) ={}&2\Re \int_{0}^{\infty}\langle (-A)^qT(t)\dot{x}(s),(-A)^qT(t)x(s)\rangle\,\mathrm{d}t \notag\\
  ={}&2\Re \int_{0}^{\infty}\langle AT(t)(-A)^qx(s),T(t)(-A)^q x(s)\rangle\,\mathrm{d}t \nonumber\\
  &\quad  +2\Re  \int_{0}^{\infty}\langle T(t)(-A)^qBu(s), (-A)^qT(t)x(s)\rangle\,\mathrm{d}t \nonumber\\
  \leq{}& -2\|(-A)^qx(s)\|^{2}+2 \int_{0}^{\infty}\|T(t)(-A)^qBu(s)\| \| (-A)^qT(t)x(s)\| \mathrm{d}t  \notag\\
  \leq{}& -\fsc{2\|(-A)^qx(s)\|^{2} + C \|u(s)\| \|(-A)^{q}x(s)\|}  \notag\\
  \leq{}& \fsc{(-2+\varepsilon)\|(-A)^{-q}\|^{-2} \|x(s)\|^{2} }+  \frac{1}{4\varepsilon}C^{2}\|u(s)\|^{2},\notag 
 \end{align}
 where $\varepsilon>0$ and \amc{$C:=2 \int_{0}^{\infty}\|T(t)\| \|(-A)^{q+\gamma}T(t)\|\|(-A)^{-\gamma}B\|_{L(U,X)} \mathrm{d}t$,}
\fsc{which is finite by \eqref{eq:Fractional-powers-analytic-semigroup} and the assumption on $B$.}
Choosing $\varepsilon=1$ and integrating yields
 \begin{equation}\label{eq:intdiss}
 W_{q}(x(h))-W_{q}(x(\tau))\leq \int_{\tau}^{h}-\|(-A)^{-q}\|^{-2}\|x(s)\|^{2}+ 4C^{2}\|u(s)\|^{2} \mathrm{d}s, \quad 0<\tau<h.
\end{equation}
By Lemma \ref{lem:Sufficient-condition-for-admissibility}, $B$ is $\max\{r,2\}$-admissible for  $r>\frac{1}{1-(\gamma+q)}$ since $B\in L(U,X_{-(\gamma+q)})$. By well-known facts on admissible control operators, \amc{\cite[Proposition 2.3]{Wei89b},  $x(\cdot)$ is continuous}, whence the inequality extends to $\tau=0$. On the other hand, $x(h)$ depends continuously on $x_0 \in X$  and $u$  with respect to the norms $X$ and $L^{\max\{r,2\}}(\Rpn,U)$. Thus, the inequality extends to $(x_{0},u)\in X\times C(\Rpn,U)$
(and even to all $X\times L^{\max\{r,2\}}(\Rpn,U)$). The dissipation inequality now follows after \amc{setting $\tau:=0$ in \eqref{eq:intdiss},} dividing by $h$, and letting $h\to0^{+}$, see  \cite[Cor. A.5.45]{CuZ20}.
\end{proof}

If, in the setting of Theorem \ref{thm:Non-Coercive-L2-ISS-LFs}, the semigroup is similar to a Hilbert space contraction semigroup, then $W_{\frac{1}{2}}$ is even coercive. 
Indeed, the fact that $W_{\frac{1}{2}}$ is well-defined on $X$ and that there exist $a_{1},a_{2}>0$ such that \eqref{eq:quadratic-sandwich} holds 
follows directly from Callier--Grabowski--Le~Merdy's theorem, see \cite[Theorem 7.3.1]{Haase06} (with $f(z)=z^{\frac{1}{2}}\mathrm{e}^{-z}$) together with \cite[Theorem 7.3.7]{Haase06}, and the references therein.
\begin{corollary}
\label{thm:Coercive-LFs-for-admissible-operators} 
Let $X$ be a Hilbert space and let $A$ generate an exponentially stable analytic semigroup $T$ on $X$, which is similar to a contraction semigroup and let $B \in L(U,X_{-\gamma})$ for some $\gamma<\frac{1}{2}$. Then $W_{\frac{1}{2}}:X\to[0,\infty)$ from \eqref{eq:LFV-non-coercive-special}
defines a coercive quadratic $L^2$-ISS Lyapunov function for \eqref{eq:Linear_System}.
\end{corollary}

{\fsc{Our results can be used to construct ISS Lyapunov functions for non-selfadjoint generators as the following example class of strongly damped wave equations show subject to viscous damping or Kelvin--Voigt damping, arising in elastic systems. Note that such systems have been intensively studied in the past, see e.g. \cite{ChenRussel82,ChenLiuLiu99,ChenTriggiani89,ChenTriggiani90,LasieckaTriggianiBook91}.
\begin{example}[Strongly damped 1D-wave equation with structural damping]
Consider the following wave equation on the spatial domain $\Omega=[0,1]$
\begin{align}
 w_{tt}(t,\xi)={}&w_{\xi\xi}(t,x)+C(w_{t}(t,x))+\delta_{\frac{1}{2}}(t)u(t), &(t,\xi)\in(0,\infty)\times \Omega, \label{eq:dampedwave1}\\
 w(t,0)={}&w(t,1)=0,&t\in(0,\infty), \label{eq:dampedwave2}\\
 w(0,\xi)={}&w_{0}(\xi),\quad w_{t}(0,\xi)=w_{1}(x),&\xi\in\Omega,
\end{align}
with a viscous damping term governed by the operator $C=A_{0}^{\alpha}:=(A_{0})^{\alpha}:D(A_{0}^{\alpha})\to L^{2}(\Omega)$, $\alpha\in[\frac{1}{2},1]$ and in-domain point-control at $x=\frac{1}{2}\in\Omega$. Here $A_{0}$ refers to the negative Dirichlet Laplacian  on $L^{2}(\Omega)$, whence $C$ is well-defined as the fractional power of a positive operator. Note aside that the extremal cases $\alpha\in\{\frac{1}{2},1\}$ refer to  $C=\partial_{\xi}$ and $C=\partial_{\xi\xi}$. Equations \eqref{eq:dampedwave1}-\eqref{eq:dampedwave2} can be phrased as a first-order system of the form \eqref{eq:Linear_System} on the state space $X=H_{0}^{1}(\Omega)\times L^{2}(\Omega)$ with
\begin{equation}
A=\begin{bmatrix}0&I\\-A_{0}&-A_{0}^{\alpha}\end{bmatrix}, \quad B=\begin{bmatrix}0\\\delta_{\frac{1}{2}}\end{bmatrix},
\end{equation}
with $D(A)=D(A_{0}^{\frac{1}{2}})\times D(A_{0}^{\alpha}) =H_{0}^{1}(\Omega)\times D(A_{0}^{\alpha})$. It is well-known that $A$ generates a strongly continuous contraction semigroup, which is even analytic and exponentially stable \cite[Thm.~1.1]{ChenTriggiani89} (also note that if $\alpha$ were in $[0,\frac{1}{2})$, then the analyticity is lost). Furthermore, by \cite[Thm.~1.1]{ChenTriggiani90}, we have for all $(\theta;\alpha)\in([0,\frac{1}{2}]\times[\frac{1}{2},1])\cup([0,1]\times\{\frac{1}{2}\})$ that 
\[D((-A)^{\theta})=D(A_{0}^{\frac{1}{2}+\theta(1-\alpha)})\times D(A_{0}^{\alpha\theta}),\]
and $D((-A^{*})^{\theta})=D((-A)^{\theta})$, where the latter is a consequence of the fact that $A_{0}$ has compact resolvent, see \cite[Remark 5.1]{ChenTriggiani90} and also \cite[Lemma A.1(vi)]{ChenTriggiani89}.
This implies that $B:\mathbb{C}\to X_{-\theta} \cong D((-A^{*})^{\theta})'$ is bounded provided that $\delta_{\frac{1}{2}}$ is an element of $D(A_{0}^{\alpha\theta})'$,  where the duality is to understood with respect to the pivot space $X$. Now, since $D(A_{0}^{\alpha\theta})$ is subspace of the Sobolev space $H^{2\alpha\theta}(\Omega)$, we conclude that if $2\alpha\theta>\frac{1}{2}$, then  indeed $\delta_{\frac{1}{2}}\in D(A_{0}^{\alpha\theta})'$, where we used that the spatial domain is one-dimensional. Thus,  we have that $B\in L(\mathbb{C},X_{-\theta})$ for  $\frac{1}{4\alpha}<\theta\leq1$. We also refer to \cite[Chapter 6]{LasieckaTriggianiBook91}, where several similar examples of this kind are treated. \newline
Hence, by Theorem \ref{thm:Non-Coercive-L2-ISS-LFs}, $W_{q}$ defines an $L^{2}$-ISS Lyapunov function for every $q$ satisfying $0\leq q\leq\min\{\frac{1}{2},1-\frac{1}{4\alpha}\}$ provided that $\alpha\in(\frac{1}{2},1]$. In the case $\alpha=\frac{1}{2}$, this holds for $q\in [0,\frac{1}{2})$. In particular, for $\alpha>\frac{1}{2}$, $W_{\frac{1}{2}}$ defines a coercive $L^{2}$-ISS Lyapunov function for the $(A,B)$. \\
Variants of this example on spatial dimensions $2$ and $3$ can be given with the (more restrictive) ranges for the parameters by simple modifications of the above reasoning.
\end{example}
}

If the semigroup (generator) is even self-adjoint, then the Lyapunov function  \eqref{eq:LFV-non-coercive-special} takes the simple form
\begin{equation}\label{eq:exprLF}
 W_{q}(x)=\frac{1}{2}\langle x,(-A)^{2q-1}x\rangle=\frac{1}{2}\|(-A)^{q-\frac{1}{2}}x\|^{2},\qquad x\in X,
\end{equation}
which is consistent with  Remark \ref{rem:LF-self-adjoint-case}. In this case, 
we get an ISS Lyapunov function also in the limiting case $p+q=1$, under an additional admissibility condition.

\begin{theorem}
\label{prop:Coercive-LF-Self-adjoint} 
Let $A=A^*$ be a self-adjoint operator, generating an exponentially stable semigroup $T$ and 
let $B \in L(U,X_{-1+q})$ be such that $(-A)^{q}B$ is $C$-admissible with $q\in[0,\frac{1}{2}]$. Then the function $W_{q}$ from  \eqref{eq:LFV-non-coercive-special}, see also \eqref{eq:exprLF}, is a quadratic $L^{2}$-ISS non-coercive Lyapunov function for \eqref{eq:Linear_System}. It is even coercive if $q=\frac{1}{2}$, in which case  $W_{\frac{1}{2}}(x)=\frac{1}{2}\|x\|^{2}$.\\
\fsc{In particular, if $\dim U<\infty$, the assumption that $(-A)^{q}B$ is $C$-admissible is automatically satisfied.}
\end{theorem}
\begin{proof}
The function $W_{q}$ is well-defined by Theorem \ref{thm:Non-Coercive-L2-ISS-LFs} and given by \eqref{eq:exprLF}. First, assume that  $x\in D(A)$ and $u$ is smooth. By self-adjointness, we get
\begin{align*}
  \frac{\mathrm{d}}{\mathrm{d}s} W_{0}(x(s)) ={}&\langle \dot{x}(s),(-A)^{2q-1}x(s)\rangle\\
            ={}&\langle Ax(s),(-A)^{2q-1}x(s)\rangle+\langle Bu(s),(-A)^{2q-1}x(s)\rangle\\
            ={}&-\|(-A)^{q}x(s)\|^{2}+\langle (-A)^{q-1}Bu(s),(-A)^{q}x(s)\rangle \\ 
            \leq{}&-\frac{1}{2}\|(-A)^{q}x(s)\|^{2}+4\|(-A)^{q-1}B\|^{2}\|u(s)\|^{2} \\
            \leq{}&-\frac{1}{2}\|(-A)^{-q}\|^{-2}\|x(s)\|^{2}+4\|(-A)^{q-1}B\|^{2}\|u(s)\|^{2}.
\end{align*}

Note that $(-A)^{q-1}B$ is a bounded operator by the assumption that $(-A)^{q}B$ is admissible.
By analogous  reasoning as in the proof of Theorem \ref{thm:Non-Coercive-L2-ISS-LFs}, we derive the dissipation inequality. Indeed, $C$-admissibility of $(-A)^{q}B$ is sufficient to extend the derived inequality to all $x_{0}\in X$ and $u\in C(\Rpn,U)$.  The coercivity of special case $q=\frac{1}{2}$ trivially follows from $W_{\frac{1}{2}}(x)=\frac{1}{2}\|x\|^{2}$.

\fsc{If $\operatorname{dim}U<\infty$ and since $A$, as a self-adjoint operator, generates a contraction semigroup, it follows by \cite[Theorem 7]{JSZ17} that all operators from $U$ to $X_{-1}$ are $C$-admissible. Therefore,  in particular, $(-A)^{q}B\in L(U,X_{-1})$ is $C$-admissible.}
\end{proof}

\begin{remark}
 The case $q=0$ is contained in a result in \cite[Theorem 5.3 and Proposition 6.1]{JMP20} which has a more technical proof. Indeed, in \cite[Proposition 6.1]{JMP20} the ISS Lyapunov function $V(x)=-\frac{1}{2}\langle A^{-1}x,x\rangle$ is considered.  By
 \begin{equation*}
  -\frac{1}{2}\langle A^{-1}x,x\rangle=\Big\langle \int_{0}^{\infty}T(2t)x\mathrm{d}t,x\Big\rangle=\int_{0}^{\infty}\langle T(t)x,T(t)x\rangle\mathrm{d}t=W_{0}(x),
 \end{equation*}
 where we used that $A$ is self-adjoint, and that the semigroup is exponentially stable, we arrive at the same Lyapunov function. Aside, we note that also the more general result \cite[Theorem 5.3]{JMP20} for not necessarily selfadjoint operators $A$ can be proved along similar lines as in the proof of  Theorem \ref{prop:Coercive-LF-Self-adjoint}, see the following Theorem \ref{thm:jmp} below. Note that in \cite{JMP20},  $\infty$-admissible operators $B$ are considered, but since the considered state spaces are reflexive, $C$-admissibility and $\infty$-admissibility are equivalent, see e.g.\ \cite{arora2024}.



\end{remark}

\fsc{\begin{example}[Heat equation with Dirichlet Control]\label{ex:heat}
Let $\Omega\subset \R^{n}$, $n\in\N$, be a domain with smooth boundary $\partial \Omega$. Consider the classical heat equation with Dirichlet boundary control
\begin{align*}
 \dot{x}(\xi,t)={}&\Delta x(\xi,t),&(\xi,t)\in\Omega\times (0,\infty),\\
 x(\xi,t)={}&u(\xi,t),&(\xi,t)\in\partial \Omega\times (0,\infty),\\
 x(\xi,0)={}&x_{0}(\xi), &\xi\in\Omega.
\end{align*}
It is well-known that this problem can be written in the form \eqref{eq:Linear_System} with
$X=L^{2}(\Omega)$, $U=L^{2}(\partial\Omega)$, $A=\Delta$ defined on  $D(A)=H^{2}(\Omega)\cap H_{0}^{1}(\Omega)$ and some $B:U\to X_{-1}$ for which $B\in L(U,X_{-\gamma})$ for any $\gamma>\frac{3}{4}$, see e.g.\ \cite[Example 2.16]{Sch20} or \cite{LasieckaTriggianiBook}. Thus, by Theorem \ref{prop:Coercive-LF-Self-adjoint}, $W_{q}$, given by \eqref{eq:exprLF}, is a non-coercive $L^{2}$-ISS Lyapunov function for every $q<\frac{1}{4}$. To see this, note that $(-A)^{q}B\in L(U,X_{-1+\frac{1}{4}-q})$ implies that $(-A)^{q}B$ is $C$-admissible in this case. Further note that $W_{q}(x)=\|(-A)^{q}x\|^{2}$ can further be represented by the functional calculus of self-adjoint operators, e.g.\ as by writing $(-A)^{q}$ as a multiplication operator.
\end{example}}
The following result, in a slightly different formulation,  was essentially first proved in \cite{JMP20}. We present a short alternative proof.
\begin{theorem}\label{thm:jmp}
Let $A$ generate an exponentially stable semigroup on a Hilbert space such that the following is satisfied.
Assume that there exists a self-adjoint, bounded operator $\tilde{P}$ such that 
 \begin{enumerate}
 \item $\langle \tilde{P}x,x\rangle>0$ for all $x\in X\setminus\{0\}$,
 \item $2\Re\langle Ax,\tilde{P}x\rangle \leq -\|x\|^{2}$ for all $x\in D(A)$,
 \item 
 $\tilde{P}A:D(A)\to X$ extends to a bounded operator on $X$.
 \end{enumerate}
 Then the function 
 \[V:X\to \R, \qquad x\mapsto  \langle \tilde{P}x,x\rangle\]
  is a non-coercive Lyapunov function for $\Sigma(A,0)$ and for any $x\in X,\epsilon>0$, $u\in C^{1}([0,\infty);U)$, it holds that
  \begin{equation} \label{eq1:thm14}
  \dot{V}_{u}(x) \leq  (\epsilon-1) \|x\|^{2} + \tfrac{1}{4\epsilon}\|\tilde{P}A\|^{2}\|B\|_{ L(U,X_{-1})}^{2}\|u(0)\|^{2}.
  \end{equation}
 If, in addition, the operator $B$ is $C$-admissible, then $V$ is an $L^{2}$-ISS Lyapunov function.  
\end{theorem}
\begin{remark}
Note that the conditions on $A$ formulated in \cite{JMP20} imply the ones used in the above theorem. To see this, let $P$ satisfy the assumptions from \cite[Thm.~4.2]{JMP20} and set $\tilde{P}:=\Re P=\frac{1}{2}(P+P^{*})$. Since $PA$ extends to a bounded operator on $X$ and $\operatorname{im}P\subset D(A^{*})$ by assumption, we conclude, by the closed graph theorem, that $A^{*}P$ and thus $P^{*}A$ extend to bounded operators. Hence $A^{*}\tilde{P}$ extends to a bounded operator on $X$. The second assumption for $\tilde{P}$ follows directly from the property that 
\[ \Re\langle (A^{*}P+PA)x,x\rangle\leq-\|x\|^{2},\qquad x\in D(A),\] assumed in \cite{JMP20}. Furthermore, we point out that in \cite[Thm.~4.2]{JMP20} a stronger notion of ISS Lyapunov functions is used, allowing for inputs in $L^{\infty}(0,t;U)$. 
\end{remark}
\begin{proof}
First, we show that \eqref{eq1:thm14} holds. Recall that for $x_{0}\in D(A)$ and $u\in C^{1}(\R_{\ge0};U)$ the mild solution is indeed a classical solution, and we have that
for all $t>0$, 
\begin{eqnarray*}
\frac{\mathrm{d}}{\mathrm{d}t}V(x(t))& = & 2\Re\langle Ax(t)+Bu(t),\tilde{P}x(t)\rangle \\
      &= & 2\Re\langle Ax(t),\tilde{P}x(t)\rangle +  2\Re\langle Bu(t),\tilde{P}x(t)\rangle \\
      &\leq & -\|x(t)\|^{2} + 2\Re\langle \tilde{P}AA^{-1}Bu(t),x(t)\rangle.
\end{eqnarray*}
Now Young's inequality implies \eqref{eq1:thm14} for any $\varepsilon>0$. Taking $\varepsilon \in (0,1)$ and integrating yields
\[V(x(t))-V(x_{0})\lesssim -\int_{0}^{t}\|x(s)\|^{2}ds+\int_{0}^{t}\|u(s)\|^{2}\mathrm{d}s.\]
If $B$ is $C$-admissible, then this inequality generalizes to all $x_{0}\in X$ and $u\in C(0,t;U)$ since $x(t)$ depends continuously on $x_{0}$ and $u$. Therefore, $V$ is an $L^2$-ISS Lyapunov function.
\end{proof}

\section{Graphic summary and closing remarks}

In Figure~\ref{fig:Overview-of-results}, we describe the relationships between the main concepts used in the paper.

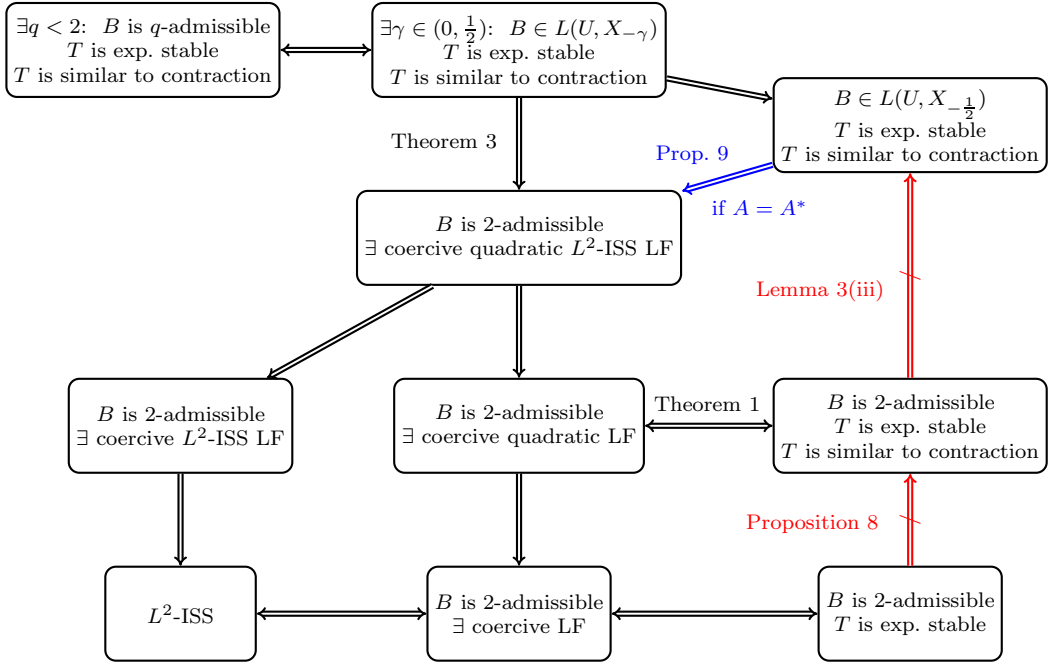
\begin{figure*}%
\usetikzlibrary{arrows}
\usetikzlibrary{graphs,decorations.pathmorphing,decorations.markings}
\usetikzlibrary{calc}

\pgfmathsetmacro\weight{1/2}
\pgfmathsetmacro\third{1/3}
\pgfmathsetmacro\twothirds{2/3}

\tikzset{degil/.style={
          decoration={markings,
          mark= at position 0.5 with {
                \node[transform shape] (tempnode) {$/$};
                 }
             },
             postaction={decorate}
}
}

\newcommand \qrq   {\quad\Rightarrow\quad}
\newcommand \qiq   {\quad\Iff\quad}

\newcommand \A   {\mathcal{A}}
\newcommand \OO   {\mathcal{O}}
\newcommand \B   {\mathcal{B}}
\newcommand \K   {\mathcal{K}}
\newcommand \Kinf{\mathcal{K_\infty}}
\newcommand \Linf  {\mathcal{L_\infty}}
\newcommand \KL  {\mathcal{KL}}
\newcommand \KKL  {\mathcal{KKL}}
\newcommand \LL  {\mathcal{L}}
\newcommand{\Knn}{\ensuremath{(\mathcal{K}_\infty\cup\{0\})^{n\times n}}}
\newcommand \PD   {\mathcal{P}}

\definecolor{GOcolor}{rgb}{0.858, 0.188, 0.478}
\definecolor{MPcolor}{rgb}{0, 0, 1}
\definecolor{HScolor}{rgb}{0.188, 0.858,  0.478}

\small
\centering
\begin{tikzpicture}[>=implies,thick, minimum height=1.25cm, minimum width=2cm]

\node (Adm-and_Lminp) at (-2.9,10.3) {Lem.~\ref{lem:Sufficient-condition-for-admissibility}};

\node[draw, rounded corners, align=center] (AnSuff-2) at (-5.5,10) {$\exists q<2$: \ $B$ is $q$-admissible\\ $T$ is exp. stable\\ $T$ is similar to contraction sg.};

\node[draw, rounded corners, align=center] (AnSuff) at (-0.2,10) {$\exists \gamma\in (0,\frac{1}{2})$: \ $B \in L(U,X_{-\gamma})$\\ $T$ is exp. stable\\ $T$ is similar to contraction sg.};

\node[draw, rounded corners, align=center] (Self-AdSuff) at (5,9) {$B \in L(U,X_{-\frac{1}{2}})$\\ $T$ is exp. stable\\ $T$ is similar to contraction sg.};

\node[draw, rounded corners, align=center] (quad-ISS-LF) at (-0.2,7.5) {$B$ is 2-admissible\\ $\exists$ coercive quadratic $L^2$-ISS LF};

\node[draw, rounded corners, align=center] (nonquad-ISS-LF) at (-5,5) {$B$ is 2-admissible\\ $\exists$ coercive $L^2$-ISS LF};

\node[draw, rounded corners, align=center] (quad-LF) at (-0.2,5) {$B$ is 2-admissible\\ $\exists$ coercive quadratic  LF};

\node[draw, rounded corners, align=center] (quad-LF-2) at (5,5) {$B$ is 2-admissible\\$T$ is exp. stable\\ $T$ is similar to contraction sg.};

\node[draw, rounded corners, align=center] (L2-ISS) at (-5,2.5) {$L^2$-ISS};

\node[draw, rounded corners, align=center] (nonquad-LF) at (-0.2,2.5) {$B$ is 2-admissible\\ $\exists$ coercive LF};

\node[draw, rounded corners, align=center] (L2-ISS-crit) at (5,2.5) {$B$ is 2-admissible\\ $T$ is exp. stable};

\draw[thick,double equal sign distance,<->] (AnSuff) to (AnSuff-2);
\draw[thick,double equal sign distance,->] (AnSuff) to (Self-AdSuff);
\draw[thick,double equal sign distance,->] (AnSuff) to (quad-ISS-LF);
\draw[color = blue, thick,double equal sign distance,->] (Self-AdSuff) to (quad-ISS-LF);
\node[color=blue] (Self-Ad) at (2.7,7.9) {if $A=A^*$};
\node (Self-Ad) at (2.2,5.3) {Thm.~\ref{thm:Quad-coercive-nonanalytic-1}};
\node[color=blue] (Self-Ad) at (1.6,8.6) {Proposition~\ref{prop:Coercive-LF-Self-adjoint}};
\node (Self-Ad) at (-1.0,8.8) {Theorem~\ref{thm:Coercive-LFs-for-admissible-operators}};
\draw[thick,double equal sign distance,->] (quad-ISS-LF) to (quad-LF);
\draw[thick,double equal sign distance,->] (quad-ISS-LF) to (nonquad-ISS-LF);
\draw[thick,double equal sign distance,<->] (quad-LF-2) to (quad-LF);

\draw[thick,double equal sign distance,->] (nonquad-ISS-LF) to (L2-ISS);

\draw[thick,double equal sign distance,->] (quad-LF) to (nonquad-LF);
\draw[thick,double equal sign distance,<->] (L2-ISS-crit) to (nonquad-LF);
\draw[thick,double equal sign distance,<->] (nonquad-LF) to (L2-ISS);

\draw[thick,double equal sign distance,->,degil,red] (L2-ISS-crit) to (quad-LF-2);
\draw[thick,double equal sign distance,->,degil,red] (quad-LF-2) to (Self-AdSuff);
\node[color=red] (TuW09) at (3.8,6.8) {Lemma~\ref{lem:Sufficient-condition-for-admissibility}(iii)};

\node[color=red] (TuW09) at (3.8,3.7) {Proposition~\ref{prop1}};

\end{tikzpicture}
\caption{$L^2$-ISS, $L^2$-ISS Lyapunov functions, and admissibility for analytic linear systems}%
\label{fig:Overview-of-results}%
\end{figure*}

In this work, we have shown that  quadratic Lyapunov functions are a natural Lyapunov function concept to study the 
$L^2$-ISS of linear analytic systems in Hilbert spaces, \amc{in the sense that the such Lyapunov functions exist and certify $L^2$-ISS for broad class of systems of this type.}
Under the assumption that $A$ generates an analytic semigroup that is similar to a contraction, and that $B\in L(U,X_{-\frac{1}{2}+\varepsilon})$ for $\varepsilon>0$, we give an explicit construction of a coercive $L^2$-ISS Lyapunov function for such a system, relying on well-known techniques from holomorphic functional calculus. If $A$ is self-adjoint, then this function is just the common Lyapunov function $V(x)=\|x\|^{2}$, which even remains coercive under the weaker assumption
that $B\in L(U,X_{-\frac{1}{2}})$. 
Note that if $B$ is not 2-admissible, then no \amc{coercive} quadratic $L^p$-ISS Lyapunov function can be constructed for this system. 

One of the interesting open problems is whether the heat equation with Dirichlet boundary input (which is an ISS system)\fsc{, see also Example \ref{ex:heat},} possesses a coercive ISS Lyapunov function. Although this question is outside of the scope of this paper, our results do shed some light on this problem. 
In particular, it is well-known that the Dirichlet input operator for the heat equation is not 2-admissible, and in fact, it is only $p$-admissible for $p>4$, see e.g.\ \cite{JPP14} and \cite{Preussler23}. This implies that there is no coercive $L^2$-ISS Lyapunov function for this system, and Proposition~\ref{prop:Restrictions-quadratic-LFs} indicates that there is also no coercive quadratic $L^p$-ISS Lyapunov function for any finite $p$. Thus, if there is a coercive ISS Lyapunov function for this system, then it is very likely not  quadratic.
Recall, that Neumann input operator $B$ for the heat equation is $\frac{4}{3}$-admissible, see e.g.\ \cite[Ex.~2.14]{Sch20}, \cite{Preussler23} or \cite[Thm.~3.2]{JPP14} and, in particular, $B\in L(U,X_{-\frac{1}{2}})$, see also Lemma~\ref{lem:Sufficient-condition-for-admissibility} above. Proposition~\ref{prop:Coercive-LF-Self-adjoint} thus
shows that $x\mapsto \|x\|^2$ is an $L^2$-ISS Lyapunov function for the heat equation with the Neumann boundary input, which is known and was also verified directly in \cite{ZhZ18} and \cite{Sch20}.

\section{Declarations}


\noindent
\textbf{Ethical Approval.} 
Not applicable.\\[-7pt]

\noindent
\textbf{Funding.} 
A.~Mironchenko is supported by the Heisenberg program of the German Research Foundation (DFG) via the grant MI 1886/3-1.\\[-7pt]

\noindent
\textbf{Authors' contributions.} 
The paper was collaboratively and equally written by both coauthors.
The order of authors is alphabetical.\\[-7pt]

\noindent
\textbf{Competing interests.} 
Both authors have no competing interests.\\[-7pt]

\noindent
\textbf{Availability of data and materials.} 
Data sharing is not applicable to this article as no datasets were generated or analysed during
the current study.\\[-7pt]

\bibliographystyle{abbrv}
\bibliography{MyPublications,Mir_LitList_NoMir}



\end{document}